\documentclass[11pt,a4paper]{amsart}%
\usepackage{amsfonts,amsmath,amssymb}
\usepackage{amssymb,amsthm,amsxtra}
\usepackage[usenames]{color}
\usepackage{amscd}
\usepackage{amsthm}
\usepackage{amsfonts}
\usepackage{amssymb}
\usepackage{amsmath}
\usepackage{graphicx}
\usepackage{hyperref}
\usepackage{enumerate}%
\usepackage[all]{xy}
\usepackage[usenames,dvipsnames]{xcolor}
\usepackage{mathrsfs}
\usepackage[margin=1.45in]{geometry}
\usepackage{cleveref}
\usepackage{stackrel}

 \newtheorem*{corollary*}{Corollary}
 \newtheorem*{construction*}{Construction}
 \newtheorem*{definition*}{Definition}
 \newtheorem*{notation*}{Notation}
 \newtheorem*{lemma*}{Lemma}
 \newtheorem*{theorem*}{Theorem}
 \newtheorem*{remark*}{Remark}
 \newtheorem*{example*}{Example}
 \newtheorem*{conjecture*}{Conjecture}
 \newtheorem*{condition*}{Condition}
 \newtheorem*{result*}{Result}
 \newtheorem*{property*}{Property}

 \newtheorem*{cor*}{Corollary}
 \newtheorem*{const*}{Construction}
 \newtheorem*{defn*}{Definition}
 \newtheorem*{notn*}{Notation}
 \newtheorem*{lem*}{Lemma}
 \newtheorem*{thm*}{Theorem}
 \newtheorem*{rem*}{Remark}
 \newtheorem*{exm*}{Example}
 \newtheorem*{conj*}{Conjecture}

 \newtheorem{lemma}{Lemma}[section]
 
 \newtheorem{remark}[lemma]{Remark}

 \newtheorem{thm}[lemma]{Theorem}
 \newtheorem{prop}[lemma]{Proposition}
 \newtheorem{lem}[lemma]{Lemma}
 \newtheorem{defn}[lemma]{Definition}
 \newtheorem{notn}[lemma]{Notation}
 \newtheorem{cor}[lemma]{Corollary}

 \crefname{introtheorem}{theorem}{theorems}
 \Crefname{introtheorem}{Theorem}{Theorems}
  
   \crefname{introthm}{theorem}{theorems}
 \Crefname{introthm}{Theorem}{Theorems}

  \crefname{introcorollary}{corollary}{corollaries}
 \Crefname{introcorollary}{Corollary}{Corollaries}

   \crefname{introcor}{corollary}{corollaries}
 \Crefname{introcor}{Corollary}{Corollaries}
 
   \crefname{introconjecture}{conjectures}{conjectures}
 \Crefname{introconjecture}{Conjecture}{Conjectures}
 
    \crefname{introconj}{conjectures}{conjectures}
 \Crefname{introconj}{Conjecture}{Conjectures}

     \crefname{introlem}{lemma}{lemmas}
 \Crefname{introlem}{Lemma}{Lemmas}
 
 \crefname{introremark}{remark}{remarks}
 \Crefname{introremark}{Remark}{Remarks}
 
  \crefname{introrem}{remark}{remarks}
 \Crefname{introrem}{Remark}{Remarks}

   \crefname{introprop}{Proposition}{Propositions}
 \Crefname{introprop}{Proposition}{Propositions}

   \crefname{introdefn}{definition}{definitions}
 \Crefname{introdefn}{Definition}{Definitions}
 
   \crefname{intronotn}{notation}{notations}
 \Crefname{intronotn}{Notation}{Notations}
 
   \crefname{introtask}{task}{tasks}
 \Crefname{introtask}{Task}{Tasks}
 
  \crefname{introprob}{problem}{problems}
 \Crefname{introprob}{Problem}{Problems}
 
   \crefname{introquestion}{question}{questions}
 \Crefname{introquestion}{Question}{Questions}

 \crefname{theorem}{theorem}{theorems}
 \Crefname{theorem}{Theorem}{Theorems}
  \crefname{thm}{theorem}{theorems}
 \Crefname{thm}{Theorem}{Theorems}

  \crefname{corollary}{Corollary}{Corollaries}
 \Crefname{corollary}{Corollary}{Corollaries}

   \crefname{cor}{Corollary}{Corollaries}
 \Crefname{cor}{Corollary}{Corollaries}

   \crefname{conjecture}{conjectures}{conjectures}
 \Crefname{conjecture}{Conjecture}{Conjectures}

    \crefname{conj}{conjectures}{conjectures}
 \Crefname{conj}{Conjecture}{Conjectures}

     \crefname{lem}{lemma}{lemmas}
 \Crefname{lem}{Lemma}{Lemmas}

      \crefname{lemma}{Lemma}{Lemmas}
 \Crefname{lemma}{Lemma}{Lemmas}

 \crefname{remark}{remark}{remarks}
 \Crefname{remark}{Remark}{Remarks}

  \crefname{rem}{remark}{remarks}
 \Crefname{rem}{Remark}{Remarks}

   \crefname{rem}{remark}{remarks}
 \Crefname{rem}{Remark}{Remarks}

   \crefname{proposition}{Proposition}{Proposition}
 \Crefname{proposition}{Proposition}{Proposition}

    \crefname{prop}{Proposition}{Propositions}
 \Crefname{prop}{Proposition}{Propositions}

   \crefname{defn}{definition}{definitions}
 \Crefname{defn}{Definition}{Definitions}

   \crefname{notn}{notation}{notations}
 \Crefname{notn}{Notation}{Notations}

   \crefname{task}{task}{tasks}
 \Crefname{task}{Task}{Tasks}

  \crefname{prob}{problem}{problems}
 \Crefname{prob}{Problem}{Problems}

   \crefname{question}{question}{questions}
 \Crefname{question}{Question}{Questions}

\newcommand{\alp}{\alpha}

\newcommand{\eps}{\varepsilon}

\newcommand{\Id}{\operatorname{Id}}

\renewcommand{\Im}{\operatorname{Im}}

\newcommand{\GL}{\operatorname{GL}}

\newcommand{\gl}{{\mathfrak{gl}}}
\newcommand{\sll}{{\mathfrak{sl}}}

\newcommand{\Sym}{\operatorname{Sym}}

\newcommand{\WF}{\operatorname{WF}}

\newcommand{\Supp}{\operatorname{Supp}}

\newcommand{\bC}{\mathbb{C}}

\newcommand{\bR}{\mathbb{R}}

\newcommand{\R}{\mathbb{R}}

\newcommand{\Sc}{\cS}

\providecommand{\fg}{\mathfrak{g}}

\providecommand{\fs}{\mathfrak{s}}

\providecommand{\fS}{\mathfrak{S}}
\providecommand{\fG}{\mathfrak{G}}

\providecommand{\cG}{\mathcal{G}}

\providecommand{\cO}{\mathcal{O}}

\providecommand{\cS}{\mathcal{S}}

\newcommand{\simpAr}[2][r]{%
\ar@{}[#1]|-*[@]_{#2}%
}

\def\smashedst{\setbox0=\hbox{$\rightrightarrows$}\ht0=4pt\box0}

\def\st#1#2{\stackrel[#2]{#1}{\smashedst}}

\begin{document}

\title{Invariant generalized functions supported on an orbit}
\author{Avraham Aizenbud}
%\address{Avraham Aizenbud, Faculty of Mathematics
%and Computer Science, The Weizmann Institute of Science POB 26,
%Rehovot 76100, ISRAEL.}
\address{Avraham Aizenbud,
Faculty of Mathematics and Computer Science, Weizmann
Institute of Science, 234 Herzl Street, Rehovot 7610001 Israel}
\email{aizenr@gmail.com}
\urladdr{http://www.wisdom.weizmann.ac.il/~aizenr}
\author{Dmitry Gourevitch}
\address{Dmitry Gourevitch, The Incumbent of Dr. A. Edward Friedmann Career Development Chair in Mathematics, Faculty of Mathematics and Computer Science, Weizmann
Institute of Science, 234 Herzl Street, Rehovot 7610001 Israel}
\email{dimagur@weizmann.ac.il}
\urladdr{http://www.wisdom.weizmann.ac.il/~dimagur}

%\subjclass[2010]{20G05, 20G25, 22E35, 46F99}
%
%\classification{20G05,20G25,46F99}
%
%20-xx  Group theory and generalizations 20Cxx Representation
%theory of groups 20C99 None of the above, but in this section
%
%20Gxx Linear algebraic groups (classical groups) 20G05
%Representation theory 20G25 Linear algebraic groups over local
%fields and their integers
%
%
%22-xx  Topological groups, Lie groups 22Exx Lie groups 22E45
%Representations of Lie and linear algebraic groups over real
%fields: analytic methods {For the purely algebraic theory, see
%20G05}
%
%
%46-xx  Functional analysis 46Fxx Distributions, generalized
%functions, distribution spaces 46F10 Operations with distributions
%
%14-xx  Algebraic geometry 14Lxx Algebraic Groups 14L24 Geometric
%invariant theory [See also 13A50] 14L30 Group actions on varieties
%or schemes (quotients) [See also 13A50, 14L24]
%22E35          Analysis on $p$-adic Lie groups
\date{\today}
\maketitle

\begin{abstract}

We study the space of invariant generalized functions supported on an orbit of the action of a real algebraic group on a real algebraic manifold. This space is equipped with the Bruhat filtration. We study the generating function of the dimensions of the filtras, and give some methods to compute it. To illustrate our methods we compute those generating functions for the adjoint action of $\GL_3(\bC)$. Our main tool is the notion of generalized functions on a real algebraic stack, introduced recently in \cite{Sak}.
\end{abstract}

\tableofcontents

\section{Introduction}\label{sec:intro}
The study of invariant distributions plays important role in representation theory and related topics (see e.g. \cite{HCBul,HCReg,GK,Sha,BerP,JR,Bar,AGRS,AG_HC,AG_MOT,SZ}). In many cases this study can be reduced to the consideration of distributions supported on a single orbit (see e.g. \cite[\S 1.5]{BerP}, \cite{KV}, \cite[Appendix D]{AG_HC}, \cite[Appendix B]{AG_ST}). While  for  non-Archimedean fields this case is very simple, for  Archimedean fields it is much more involved. In this paper we establish some infrastructure in order to analyze the Archimedean case.

Let a Nash\footnote{Nash manifolds are generalizations of real algebraic manifolds. In most places in this paper the reader can safely replace the word Nash by ``smooth real algebraic''. For more details on Nash manifolds and Schwartz functions over them see \cite{AG}.} group $G$ act on a Nash manifold $M$. Let $O$ be an  orbit of $G$ in $X$. The space $\cG(X\smallsetminus(\bar O\smallsetminus O))^G$ of tempered $G$-invariant generalized functions defined in a neighborhood of $O$ and supported in $O$ is equipped with the Bruhat filtration (see e.g. \cite{AG}).
Let $\bar\delta_{\cO}^X(i)$ denote the dimension of the $i$-s filtra and $$\bar{\mathfrak{G}}_{\cO}^X(t):=(1-t)\sum_i t^i\bar\delta_{\cO}^X(i)$$  denotes the corresponding generating function.

In this paper we introduce several techniques for the computation of this function. We illustrate our methods on the case of  the adjoint action of $\GL_3(\bC)$. Our main tool is the notion of generalized functions on a real algebraic stack, introduced recently in \cite{Sak}.
\subsection{Results}
\begin{enumerate}
\item In the case of when $O$ is (locally) a fiber of a $G$-invariant submersion we prove that $\bar \fG_{O}^X(t)=(1-t)^{\dim O-\dim X}$ (see Corollary \ref{cor:reg}). \item We prove that  $\bar\delta_{\cO}^X(i)-\bar\delta_{\cO}^X(i-1)$ is bounded by  $\dim (\Sym^i(N_{\cO,x}^X))^{G_x}$ (see Lemma \ref{lem:obv}), and in the case when the stabilizer of a point in $O$ is reductive, this bound is achieved (see Theorem \ref{thm:HC})

\item  We prove that  $\bar \fG_{O}^X(t)$ is multiplicative in an appropriate sense (see Lemma \ref{lem:obv}).
\item In the general case we reduce the computation of $\bar \fG_{O}^X(t)$ to the computation of certain subspace of  distributions supported on a point in a manifold of dimension $\dim X-\dim O$ (see Theorem \ref{thm:main}). Under certain connectivity assumptions this can be reduced to an infinite dimensional linear algebra problem (see Corollary \ref{cor:main}).
\item For the case of  the adjoint action of $\GL_3(\bC)$ on its lie algebra (or equivalently on itself) we compute $\bar \fG_{O}^X(t)$ for all orbits (see \S \ref{sec:Exm}).
\end{enumerate}

\subsection{Ideas in the proof}
Results (2,3) follows easily from the existing knowledge on invariant distributions. Result (1) follows easily from (4). Result (5) is a computation based on (4).
In order to formulate and prove Result (4) we use \cite{Sak}. Namely  we find a different presentation of the quotient stack  $G\backslash X$, and use the fact that the space of generalized functions on a stack does not depend on the presentation (See \cite[Theorem 3.3.1]{Sak}). In order to compute generalized functions in the new presentation we  replace our groupoid structure by an infinitesimal one. We do it in Theorem \ref{thm:groupoid}.

\subsection{Structure of the paper}

In \S \ref{sec:Prel} we fix notation for generalized functions on Nash manifolds, Nash groupoids and Nash stacks.

In \S \ref{sec:Group} we analyze generalized functions on groupoids. We prove Theorem \ref{thm:groupoid} which states that, under certain continuity assumptions, generalized functions on a groupoid are generalized functions on the objects manifold, satisfying a certain system of PDE.

In \S \ref{sec:GrowthFn} we define the function $\bar{\mathfrak{G}}_{\cO}^X(t)$, which is the main object of study in this paper, and establish its basic properties.

In \S \ref{sec:Nash} we introduce the stack slice, which is our main geometric tool for the computation of $\bar{\mathfrak{G}}_{\cO}^X(t)$.

In \S \ref{sec:Orb} we present a method to compute $\bar{\mathfrak{G}}_{\cO}^X(t)$ using the stack slice. We implement this method for regular orbits.

In \S \ref{sec:Exm} we compute $\bar{\mathfrak{G}}_{\cO}^X(t)$ for the adjoint action of $\GL_3(\bC)$.

\subsection{Acknowledgements}
We thank Joseph Bernstein, Bernhard Kroetz and Siddhartha Sahi for motivating questions, and Yiannis Sakellaridis for explaining us his work \cite{Sak}.

A.A.  was partially supported by ISF grant 687/13.
D.G. was partially supported by ERC StG grant 637912, and ISF grant 756/12.

\section{Preliminaries on generalized functions}\label{sec:Prel}
In this section we fix some notation %and give a short overview on
concerning generalized functions on manifolds, and tempered generalized functions on Nash manifolds and Nash stacks. We refer the reader to \cite{Hor,AG,Sak} for more details.

For a smooth manifold $M$ we denote by $C^{-\infty}(M)$ the space of generalized functions, i.e. continuous functionals on the space of compactly supported smooth measures. If $M$ has a fixed smooth invertible measure then this space can be identified with the space of distributions on $M$.

For a smooth real algebraic manifold (or, more generally, a Nash manifold) $M$ we denote by $\Sc(M)$ the space of Schwartz functions on $M$ (see e.g. \cite{AG}), and by $\Sc^*(M)$  the dual space. We call the elements of  $\Sc^*(M)$ \emph{tempered distributions} (Schwartz distributions in \cite{AG}). We also denote by $\cG(M)$ the space of tempered generalized functions, i.e. functionals on the space of Schwartz measures $\Sc(M,D_M)$ (see \cite{AG}).

For a distribution or a generalized function $\xi$ on a manifold $M$ we denote by $\WF(\xi)$ its wave-front set (see \cite[\S 8.1]{Hor}).

\begin{defn}$\,$
%\begin{enumerate}[(i)]
A Lie (resp. Nash) groupoid is a diagram $\{Mor \st{s}{t} Ob\}$ of smooth (resp. Nash) manifolds such that $s$ and $t$ are submersions,  a smooth (resp. Nash) composition map  $\mathrm{comp}: Mor\times_{Ob}Mor\to Mor$,  a smooth (resp. Nash) identity section $\mathrm{I}:Ob \to Mor$ and a smooth (resp. Nash) inversion map $inv:Mor\to Mor$ satisfying the usual groupoid axioms.
%\end{enumerate}
\end{defn}

\begin{defn}
 A generalized function $\xi\in C^{-\infty}(S)$ on a Lie groupoid $S=\{Mor \st{s}{t} Ob\}$ is a generalized function on $Ob$ such that $t^*\xi=s^*\xi$.
If $S$ is a Nash groupoid, we also define the space $\cG(S)$ of tempered generalized functions in a similar way.
\end{defn}
In \cite[Theorem 3.3.1]{Sak} it is shown that $\cG(S)$ depends only on the Nash stack corresponding to $S$ (see \cite[\S 2.2]{Sak} for the definition of the Nash stack corresponding to a Nash groupoid). Note that \cite{Sak} uses the notation $\Sc$ for Schwartz measures and $\Sc^*$ for generalized functions.

\section{Generalized functions on smooth groupoids}\label{sec:Group}
\setcounter{lemma}{0}

\begin{thm}\label{thm:groupoid}
Let $S=\{Mor \st{s}{t} Ob\}$ be a Lie groupoid. Let $\xi\in C^{-\infty}(Ob)$. Consider the following properties of $\xi$:
\begin{enumerate}
\item \label{it:xiS} $\xi \in C^{-\infty}(S).$
\item \label{it:open}For any open subset $U\subset Ob$ and any section $\varphi:U\to Mor$ of $s$ such that $\psi:=t\circ \varphi:U \to Ob$ is an open embedding we have $\psi^*\xi=\xi|_U$.
%\item \label{it:emb} For  any open subset $U\subset Ob$ and any open embedding %$\psi:U\into Ob$ such that for any $x\in U$ there is a morphism $m:x \to %\psi(x)$, we have $\psi^*\xi=\xi|_U$.

\item \label{it:temp} For any $m\in Mor,$ there exist smooth manifolds $U,V$ and a submersion $\varphi:V\times U \to Mor$ with $m\in \Im \varphi$ such that for any $x\in V$, the maps $\varphi_x^s:=s\circ\varphi|_{\{x\}\times U}$ and  $\varphi_x^t:=t\circ\varphi|_{\{x\}\times U}$ are open embeddings and we have $(\varphi_x^t)^*\xi=(\varphi_x^s)^*\xi.$

\item \label{it:VF} For any section $\alp$ of $I^*TMor$, where $I:Ob\to Mor$
is the identity section, with $ds(\alp)=0$ we have $dt(\alp)\xi=0$. Here, $dt(\alp)$ and $ds(\alp)$ are the vector fields given by $ds(\alp)_x:=d_{Id_x}s(\alp_x)$, $dt(\alp)_x:=d_{Id_x}t(\alp_x).$  \end{enumerate}
Then \eqref{it:xiS}$\Leftrightarrow$\eqref{it:open}$\Leftrightarrow$\eqref{it:temp}$\Rightarrow$\eqref{it:VF} and if for all $x\in Ob,\, s^{-1}(x)$ is connected then \eqref{it:temp}$\Leftrightarrow$\eqref{it:VF}.
\end{thm}

For the proof we will need the following lemmas.

\begin{lem}\label{lem:prod}
Let $X,Y$ be smooth manifolds. Let $\xi\in C^{-\infty}(X\times Y)$ such that for any $x\in X$, $\WF(\xi)\cap CN_{\{x\}\times Y}^{X\times Y}\subset \{x\}\times Y$ and $\xi|_{\{x\}\times Y} = 0$. Then $\xi=0$.

Here, the restriction $\xi|_{\{x\}\times Y}$ is in the sense of \cite[Corollary 8.2.7]{Hor}.
\end{lem}

This lemma follows from the next one in view of \cite[Theorem 8.2.4 and the proof of Theorem 8.2.3]{Hor}.

\begin{lem}\label{lem:prod2}
Let $V=\R^n,W=\R^k$ be real vector spaces. Let $\xi\in C^{-\infty}(V\times W)$ such that for any $x\in V$, $\WF(\xi)\cap CN_{\{x\}\times W}^{V\times W}\subset \{x\}\times W$. Fix  Lebesgue measures $V$ and $W$.
Let $f\in C_c^{\infty}(V\times W)$. Let $e_i\in  C_c^{\infty}(V\times W)$ be a sequence satisfying $\int_{V\times W} e_i(z) dz = 1$ and $e_i(z)=0$ for any $z$ with $||z||>1/i$.
For any $x\in V$ denote $g(x):=\langle \xi|_{\{x\}\times W},f|_{\{x\}\times W}\rangle$ and $g_n(x):=\langle (\xi*e_n)|_{\{x\}\times W},f|_{\{x\}\times W}\rangle$.
 Then $g_n{\to} g$ uniformly as ${n\to \infty}$.
\end{lem}
\begin{proof}
Let $U\supset \Supp f$ be an open centrally symmetric set with compact closure.
Let
\begin{multline*}\Gamma:=\left (V \times pr_{W\times (V\times W)^*}\left(\WF(\xi) \cap (\bar U\times (V\times W)^*)\right)\right)\cup \left (((V \times W) \smallsetminus U) \times (V\times W)^*\right)\\\subset  T^*(V \times W).
\end{multline*}
For any $x\in V$ denote $\xi_x:=Sh_x(\xi)$, where $Sh_x$ is the translation by $x$.
Denote $$C^{-\infty}_{\Gamma}(V\times W):=\{\eta\in C^{-\infty}(V\times W), \, \WF(\eta) \subset \Gamma\},$$
with the topology of \cite[Definition 8.2.2]{Hor}.
It is easy to see that $x \mapsto \xi_x$ defines a continuous map $V \to C^{-\infty}_{\Gamma}(V\times W)$. Let
$\xi_{n,x}:=\xi_x *e_n.$ The proof of \cite[Theorem 8.2.3]{Hor} implies that $\xi_{n,x} \to \xi_x$ as $n \to \infty$ in the topology of $C^{-\infty}_{\Gamma}(V\times W)$ uniformly in $x$. Thus, by \cite[Theorem 8.2.4]{Hor}, $\xi_{n,x}|_{\{0\}\times W}\to \xi_x|_{\{0\}\times W}$ as $n \to \infty$ in the weak topology of $C^{-\infty}(W)$ uniformly in $x$. This implies the assertion.
\end{proof}

The following lemma is standard.
\begin{lem}
Let $\varphi:M\to N$ be a submersion of smooth manifolds with connected fibers.
Let $s_0,s_1:N\to M$ be its (smooth) sections. Then, for any $y\in N$, there exists an open neighborhood $U$ of $y$ and a smooth homotopy $h:[0,1]\times U \to M$ such that $h|_{\{0\}\times U}=s_0$, $h|_{\{1\}\times U}=s_1$, and $h|_{\{t\}\times U}$ is a section of $\varphi$ for any $t$.
\end{lem}

\begin{cor}\label{cor:hom}
Let $\varphi_1:M_1\to N$ and $\varphi_2:M_2\to N$ be submersions of smooth manifolds. Assume that all the fibers of $\varphi_2$ are connected. Let $\psi_1,\psi_2:M_1\to M_2$ be smooth maps of $N$-manifolds (that is, smooth maps such that $\varphi_2\circ \psi_i=\varphi_1$).
 Then, for any $y\in M_{1}$, there exists an open neighborhood $U$ of $y$ and a smooth homotopy $h:[0,1]\times U \to M_{2}$ such that $h|_{\{0\}\times U}=\psi_0$, $h|_{\{1\}\times U}=\psi_1$, and $h|_{\{t\}\times U}$ is a map of $N$-manifolds.

\end{cor}

\begin{proof}[Proof of \Cref{thm:groupoid}]$\,$
\begin{itemize}
\item [\eqref{it:xiS}$\Rightarrow$\eqref{it:open}:]by functoriality of the pullback.
\item [\eqref{it:open}$\Rightarrow$\eqref{it:temp}:]
It is enough to show that  for any $m\in Mor,$ there exist smooth manifolds $U,V$ and a submersion $\varphi:V\times U \to Mor$ with $m\in \Im \varphi$ such that for any $x\in V$, the maps $\varphi_x^s$ and  $\varphi_x^t$ are open embeddings.
Since $s$ and $t$ are submersions, we can decompose $T_m(Mor)=V'\oplus U'$ such that $d_ms|_{U'}$ and $d_mt|_{U'}$ are isomorphisms. Let $\varphi': T_m(Mor)\to Mor$ be such that $\varphi'(0)=m$ and $d\varphi'=Id$.  By the implicit function theorem one can choose open subsets $U\subset U'$ and $V \subset V'$ such that $\varphi:=\varphi'|_{U\times V}$ is a submersion and the maps $\varphi_x^s$ and  $\varphi_x^t$ are open embeddings.

%Let $U\subset Ob$ and let $\psi:U\into Ob$ be an open embedding  such that %for any $x\in U$ there is a morphism $m:x \to \psi(x)$. For any $x\in U$ %we can find an open subset $V\subset \varphi(U)$ that includes $x$ and a %section $\lambda:V \to Mor$ of $t$ such that $s(\lambda(\psi(x)))=x$. Let %$\varphi:=\lambda \circ \psi|_{\psi^{-1}(V)}$.

\item [\eqref{it:temp} $\Rightarrow$  \eqref{it:xiS}:] by \Cref{lem:prod}.
\item [\eqref{it:open} $\Rightarrow$ \eqref{it:VF}:] Let  $\alp$ be a section of $I^*TMor$ with $ds(\alp)=0$. Define a vector field $\beta$ on $Mor$ by $$\beta_m:=(d_{I(t(m)),m}comp)(\alp_{t(m)},0),$$
where $comp:Mor\times_{Ob}Mor\to Mor$ is the composition map. By the existence and uniqueness theorem for ODE, we have an open neighborhood $\cO$ of $Mor\times \{0\}$ in $Mor \times \R$ and a map $B:\cO\to Mor$ that solves the ODE defined by $\beta$. Fix $x\in Ob$. There exists a neighborhood $U$ of $x$ and $\eps>0$ such that $U\times (-\eps,\eps) \subset \cO$.  For any $r\in (-\eps,\eps)$ define $\varphi_r(x):=B(x,r)$. Define $\psi_r:=t\circ \varphi_r$. By \eqref{it:open} we have $\psi_r^*\xi=\xi|_U$. On the other hand, it is easy to see that $$\frac{d}{dr}|_{r=0}\psi_r^*\xi=dt(\alp)\xi|_U.$$

\item [\eqref{it:VF} $\Rightarrow$ \eqref{it:temp},] for connected $s^{-1}(x)$: Let $m\in Mor$. By \Cref{cor:hom}, there exist an open neighborhood $V$ of $m$ and a smooth homotopy $h:[0,1]\times V \to Mor$ such that
$$h|_{\{0\}\times V}=I\circ s,\,\, h|_{\{1\}\times V}=I\circ s \text{ and }s(h(r,x))=s(x).$$ For any $r\in [0,1]$ and $u\in V$ consider
%$\overset{.}{\gamma}\in T_{\gamma(r)}Mor$.
$$\alp(r,v):=d_{h(r,v),inv(h(r,v))}comp(\frac{d}{dr}h(r,v),0)\in T_{I(t(h(r,v)))}Mor.$$
Extend $\alp(r,v)$ to a smooth section of $I^*TMor$ in a way that depends smoothly on $(r,v)$  such that $ds(\alp(r,v))=0$ for any $(r,v)\in [0,1]\times V $. By \eqref{it:VF} we have $dt(\alp(r,v))\xi=0$. Define a vector field $\beta(r,v)$ on $Mor$ by
\begin{equation}\label{=bet}
\beta(r,v)_n:=d_{I(t(n)),n}comp(\alp(r,v)_{t(n)},0).
\end{equation}
For any $v\in V$ we consider $\beta(\cdot,v)$ as a time-dependent vector field on $Mor$. By the existence and uniqueness theorem for ODE, we have an open neighborhood $\cO$ of $Mor\times \{0\}\times V\cup \{h(r,v),r,v) \, \vert \, r\in [0,1], \, v\in V\}$ in $Mor \times [0,1]\times V$ and a map $B:\cO\to Mor$ that solves the ODE defined by $\beta$.
Let $$\Xi:=I\times Id_{[0,1]\times V}^{-1}(\cO)\subset Ob \times [0,1]\times V,  A:=B\circ (I\times Id_{[0,1]\times V}|_{\Xi}) \text{ and }C:=t\circ A.$$
Let $U$ be a neighborhood of $s(m)$ in $Ob$ such that $U \times [0,1]\times V \subset \Xi$. Define $\varphi:U \times V \to Mor$ by $\varphi:=A|_{U \times \{1\}\times V }$. It is enough to prove that for any $x\in V$:
\begin{enumerate}[(i)]
\item \label{it:sOpEmb} The map $\varphi_x^s:=s\circ\varphi|_{\{x\}\times U}$ is an open embedding.
\item \label{it:tOpEmb} The map  $\varphi_x^t:=t\circ\varphi|_{\{x\}\times U}$ is an open embedding.
\item \label{it:tsxi} We have $(\varphi_x^t)^*\xi=(\varphi_x^s)^*\xi.$
\end{enumerate}
Note that $\varphi_x^s=Id$ and thus \eqref{it:sOpEmb} holds.

For $(v,r)\in V\times [0,1]$ let $\gamma(v,r):=dt(\alp(v,r))$ be a vector field on $Ob$. By \eqref{it:VF}, $\gamma(v,r)\xi=0$. It is easy to see that $$\frac{\partial}{\partial r}C(x,v,r)=\gamma(v,r)C(x,v,r).$$
Thus, for any $(x,v)\in Ob\times V$, the (partially defined) curve $C(x,v,\cdot)$ is a  solution of the ODE defined by the time-dependent vector field $\gamma(v,\cdot)$.
Note that $\varphi_x^t=C|_{U\times \{x\}\times \{1\}}$ and thus \eqref{it:tOpEmb} holds.

Finally, \eqref{it:tsxi} follows from the equality $\gamma(v,r)\xi=0$.
%
% Let $m\in Mor$. Let $\gamma:\R\to Mor$ be a smooth path such that $\gamma(0)=I(s(m))$, $\gamma(1)=m$ and $s\circ \gamma$ is constant.    Consider
% %$\overset{.}{\gamma}\in T_{\gamma(r)}Mor$.
% $$\alp(r):=d_{\gamma(r),inv(\gamma(r))}comp(\overset{.}{\gamma}(r),0)\in T_{I(t(\gamma(r)))}Mor.$$
% Extend $\alp$ to a time-dependent section of $I^*TMor$ such that $ds(\alp)=0$. Define a vector field $\beta$ on $Mor$ by $$\beta(r)_m:=d_{I(t(m)),m}comp(\alp(r)_{t(m)},0).$$
% By the existence and uniqueness theorem for ODE, we have an open neighborhood $\cO$ of $Mor\times \{0\}\cup \{(\gamma(t),t) \, \vert \, t\in \R\}$ in $Mor \times \R$ and a map $B:\cO\to Mor$ that solves the ODE defined by $\beta$.
% By \eqref{it:VF} we have $dt(\alp(r))\xi=0$. For any $m\in Mor$,
% $$d_mt(\beta(r)_m)=d_{I(t(m),m)}(\alp(r)_{t(m)}).$$
% Thus $\beta(r)(t^*\xi)=0$.
% ?? thus the evolution does not change $\xi$, but at 1 it is not an embedding. Probably we should have extended the vector field in a non-arbitrary way.
\end{itemize}
\end{proof}

% \begin{prop}
% Let $S=\{Mor \st{s}{t} Ob\}$ be a Nash groupoid and let $\fS$ be the corresponding Nash stack. Then $\cG(\fS)\cong C^{-\infty}(S)\cap \cG(Ob)$.
% \end{prop}

\section{The dimension growth function of an orbit}\label{sec:GrowthFn}
Let a Nash group $G$ act on a Nash manifold $X$.
Let $\cO\subset X$ be an orbit. Let $F_i$ be the Bruhat filtration on $\cG_{\cO}(X\smallsetminus(\bar \cO\smallsetminus \cO))$ (see \cite[Corollary 5.5.4]{AG}). Let $$V_i:=\{\xi \in F_{i} \, \vert \, \exists \eta\in \cG(X) \text{ s.t. }\eta|_{(X\smallsetminus(\bar \cO\smallsetminus \cO))}=\xi\}.$$
Define the distributional dimension growth function of $\cO$ in $X$ by $$\delta_{\cO}^X(i):=\dim V_i.$$
Define also the distributional normal dimension of $\cO$ in $X$ by
$$Ddim(\cO,X):=\limsup_i \frac{\ln \delta_{\cO}^X(i)}{\ln i},$$
and the distributional normal degree of $\cO$ in $X$ by
$$Ddeg(\cO,X):=\limsup_i \left(Ddim(\cO,X)! \delta_{\cO}^X(i){i^{-Ddim(\cO,X)}}\right).$$

Define  the  distributional dimension generating function by $$\mathfrak{G}_{\cO}^X(t):=(1-t)\sum_i t^i\delta_{\cO}^X(i).$$

Finally, define  the reduced versions of the above notions by
 $$\bar\delta_{\cO}^X:=\delta_{\cO}^{X\smallsetminus(\bar \cO\smallsetminus \cO)} , \, \overline{Ddim}(\cO,X):=Ddim(\cO,X\smallsetminus(\bar \cO\smallsetminus \cO)), \, \overline{Ddeg}(\cO,X):=Ddeg(\cO,X\smallsetminus(\bar \cO\smallsetminus \cO))$$

For a point $x\in \cO$ we will denote
\begin{multline*}\delta_{x}^X(i):=\delta_{\cO}^X(i), \quad \bar \delta_{x}^X(i):=\bar \delta_{\cO}^X(i), \quad  Ddim(x,X):=Ddim(\cO,X), \\ \overline{Ddim}(x,X):=\overline{Ddim}(\cO,X),\quad Ddeg(x,X):=Ddeg(\cO,X), \quad \overline{Ddeg}(x,X):=\overline{Ddeg}(\cO,X).
\end{multline*}

The following lemma follows from \cite[Corollary 5.5.4]{AG} and \cite[Corollary 2.6.3]{AG_DR}.
\begin{lemma}\label{lem:obv}
Let $x\in \cO$.  Then
\begin{enumerate}[(i)]
\item
$\delta_{\cO}^X(i)-\delta_{\cO}^X(i-1)\leq \dim (\Sym^i(N_{\cO,x}^X))^{G_x}$ and $Ddim(\cO,X)\leq\dim (N_{\cO,x}^X)$.
\item Let $U$ be an open $G$-invariant neighborhood of $\cO$ in $X$. Then $$\delta_{\cO}^U(i)\geq \delta_{\cO}^X(i) \text{ and }Ddim(\cO,U)\geq Ddim(\cO,X).$$
\item \label{it:prod} Let another Nash group $G'$ act on a Nash manifold $X'$, and $\cO'$ be an orbit. Consider the action of $G\times G'$ on $X \times X'$. Then
$$ \mathfrak{G}_{\cO\times \cO'}^{X\times X'}(t)=\mathfrak{G}_{\cO}^X(t)\mathfrak{G}_{\cO'}^{X'}(t).$$
\end{enumerate}
\end{lemma}

%\subsection{The case of reductive stabilizer}

The following theorem follows from the proof of \cite[Theorem 3.1.1]{AG_HC}.
\begin{thm}\label{thm:HC}
Let a reductive group $G$ act on an affine algebraic manifold $X$. Let $\cO\subset X$ be a closed orbit. Then $\delta_{\cO}^X(i)-\delta_{\cO}^X(i-1)=\dim (\Sym^i(N_{\cO,x}^X))^{G_x}.$
\end{thm}

\begin{remark}$\,$
\begin{enumerate}[(i)]
\item One can replace the assumption that $G$ is reductive and $X$ is affine by the weaker assumption that the stabilizer of a point $x\in \cO$ is reductive. For that one needs to use the version of the Luna slice theorem appearing in  \cite[Theorem 2.1]{StackLuna}.
\item Since the Poincare series of a finitely generated graded algebra is a rational function whose poles are roots of unity (see e.g. \cite[Theorem 11.1]{AM}), Theorem \ref{thm:HC} implies that $\mathfrak{G}_{\cO}^X(t)$ is  also a rational function whose poles are roots of unity.

\item In a similar way one can show that under the conditions of Theorem \ref{thm:HC}, the dimension  $Ddim(\cO,X)$ equals the dimension of the categorical quotient $N_{\cO,x}^X//G_x$ of $N_{\cO,x}^X$ by $G_x$.

\item Similarly, if in Lemma \ref{lem:obv} the categorical quotient $N_{\cO,x}^X//G_x$  exists then $Ddim(\cO,X)\leq\dim (N_{\cO,x}^X//G_x)$.

\end{enumerate}
\end{remark}

\section{Restriction of a Nash stack to a slice}\label{sec:Nash}
Let a Nash group $G$ act on a Nash manifold $X$.
\setcounter{lemma}{0}
\begin{defn}\label{def:StackSlice}
 Choose a point $x\in X$ and let $\cO:=Gx$ be its orbit.
\begin{enumerate}
\item We call a locally closed Nash submanifold $S\subset X$ a slice to the action of $G$ at $x$ if $x\in S$, the action map $a:G\times S \to X$ is a submersion, and $\dim \cO+\dim S=\dim X$.

\item Let $S$ be a slice to the action of $G$ at $x$. Define $M_S:=a^{-1}(S)\subset G\times S$.  Consider the quotient Nash groupoid $G\times X  {\st{pr}{a}}X$
and its subgroupoid $M_S \st{pr}{a}S$. We will call this subgroupoid a
\emph{groupoid slice} to the action of $G$ at $x$, and call
 the corresponding Nash stack a \emph{stack slice} to the action of $G$ at $x$.
%$x\stackrel[a]{b}{=}x$
%$X \st{pr}{a} X$
\end{enumerate}
\end{defn}

\begin{lem}\label{lem:ExSlice}
For any $x\in X$ there exists a slice to the action of $G$ at $x$.
\end{lem}
\begin{proof}
Choose a direct complement $W$ to $T_xX$ in $T_xGx$. It is a standard fact that there exists  a Nash manifold $S'\subset X$ containing $x$ such that $T_xS'=W$. Consider the action map $a:G \times S' \to X$. Let $S:=\{x\in S' \, \vert \,  a \text{ is a submersion at }(1,x)\}.$ It is easy to see that $S$ satisfies the conditions.
\end{proof}

The following proposition follows from the definition in \cite[\S 2.2]{Sak}.
\begin{prop}\label{prop:SliceEx}
For any $x\in X$ and any  stack slice $\fS$ to the action of $G$ at $x$ there exists an open Nash $G$-invariant neighborhood $U$ of $x$ and  such that $G\backslash U\cong\fS$.
\end{prop}

\section{Description of the space of invariant generalized functions supported on an orbit}\label{sec:Orb}

\Cref{prop:SliceEx} and \cite[Theorem 3.3.1]{Sak} imply the following theorem

\setcounter{lemma}{0}
\begin{thm}\label{thm:main}
Let a Nash group $G$ act on a Nash manifold $X$.
Let $x\in X$ such that the orbit $Gx$ is closed. Then for any groupoid slice  $\fS$ to the action of $G$ at $x$ we have a canonical isomorphism $\cG_{Gx}(X)^{G}\cong \cG_{\{x\}}(\fS)$. Here, we consider $\{x\}$ as a closed subset in $\fS$.
\end{thm}

\begin{notn}
Let a Lie group $G$ act on a smooth manifold $X$. Let $S\subset X$ be a (locally closed) smooth submanifold. Let $\varphi:S \to \fg$ be a smooth map.
For any $s\in S$ define   $\alp_{\varphi}(s)\in T_sX$ by $\alp_{\varphi}(s):=d_{e}(a_s(\varphi(s)))$, where $a_s:G \to X$ is the action map on $s$ and $e\in G$ is the unit element. Suppose that $\alp_{\varphi}$ defines a vector field on $S$, i.e. $\alp_{\varphi}(s)\in T_sS$ for any $s\in S$. Then we call this field \emph{strongly tangential to the action of $G$}.
\end{notn}

\Cref{thm:main,thm:groupoid} give the following corollary.

\begin{cor}\label{cor:main}
Let a Nash group $G$ act on a Nash manifold $X$. Let $x\in X$ such that the orbit $Gx$ is closed. Let  $S$ be a  slice  to the action of $G$ at $x$.
 Then we have a canonical embedding of $\cG_{Gx}(X)$ into the space$$\{\xi \in \cG_{\{x\}}(S)\, \vert \, \alp\xi=0  \text{ for any vector field } \alp \text{ on }S \text{ strongly tangential to the action of }G\}.$$
Moreover, if for all $x\in S,$ the set of all $g\in G$ with $gx\in S$ is connected then this embedding is an isomorphism.
\end{cor}
% - 2*a^3 + 3*a^2*la + 2*b*a - la^3 + b*la + d*e

% - la^3  + (3*a^2 +   b)*la + d*e - 2*a^3 +2*b*a

% [     a^2 + b,     2*a,     e]
% [ 2*a*b + d*e, a^2 + b,  -a*e]
% [        -a*d,       d, 4*a^2]

%\subsection{Regular case}

\begin{cor}\label{cor:reg}
Let $\varphi:X\to Y$ be a Nash submersion of Nash manifolds.
Let a Nash group $G$ act on $X$ preserving $\varphi$. Let $y\in Y$ and assume that $G$ acts transitively on the fiber $\varphi^{-1}(y)$. Then $\cG_{\varphi^{-1}(y)}(X)^{G}$ is isomorphic as a filtered vector space to $\bC[t_1,\dots,t_{\dim Y}]$.
In particular,
\begin{equation}
\fG_{\varphi^{-1}(y)}^X(t)=(1-t)^{-\dim Y}, \, Ddim(\varphi^{-1}(y),X)=\dim Y, \text{ and } Ddeg(\varphi^{-1}(y),X)=1.
\end{equation}
\end{cor}
\begin{proof}
Let $x\in \varphi^{-1}(y)$.
By \Cref{lem:ExSlice} there exists
a slice $S$ to the action of $G$ on $X$ at $x$. Shrinking $S$, we can assume that $\varphi|_S$ is an
% \et \,  map. By the inverse function theorem we can shrink $S$ further
open embedding.
Let  $M_S \st{pr}{a}S$ be as in \Cref{def:StackSlice}. By the assumption, $pr=a$. Thus the corollary follows from \Cref{thm:main}.
\end{proof}

%\subsection{Infinitesimal version}

%\subsection{Slodowy slice}

\section{Computation of $\bar \delta$ for the adjoint action  of $\GL_3(\bC)$}\label{sec:Exm}
\setcounter{lemma}{0}

\begin{thm}\label{thm:gl3}
Consider the adjoint action of $G:=\GL_3(\bC)$ on its Lie algebra $\fg$.
Let $x\in \fg$ and let $m_x$ denote its minimal polynomial. Then
$$\bar\fG_{x}^{\fg}(t)= \begin{cases}
(1-t)^{-6} & \deg m_x=3\\
(1-t)^{-6}(1+t)^{-4}(t^{2}-t+2)^2 &  m_x=(x-\lambda)^2 \\
(1-t)^{-6}(1+t)^{-2} & m_x=(x-\lambda)(x-\mu), \, \lambda\neq \mu \\
(1-t)^{-6}(1+t)^{-2}(1+t+t^2)^{-2}    & \deg m_x=1\\
\end{cases},$$
$\overline{Ddim}(\cO,X)=6$ and $\overline{Ddeg}(\cO,X)=((3-\deg m_x)!)^{-2}.$
\end{thm}

%?? Generalized functions versus distributions.

The case $\deg m_x=3$ follows from \Cref{cor:reg}. The case $\deg m_x=1$ follows from \Cref{thm:HC}. The case $m_x=(x-\lambda)(x-\mu), \, \lambda\neq \mu$ follows from \Cref{thm:HC} and \Cref{lem:obv}\eqref{it:prod}. Thus it is enough to prove the following proposition.
%?? Next version: expand!! 

\begin{prop}\label{prop:gl3}
Let $G:=\GL_3(\bC)$ act on $X:=\sll_3(\bC)\smallsetminus 0$ by conjugation. Let $x\in X$ be the subregular nilpotent matrix. Then
$\bar\fG_{x}^{X}(t)=(1-t)^{-4}(1+t)^{-4}(t^{2}-t+2)^2 $.
\end{prop}

Let $e:=E_{12}\in \cO$. Let $f:=E_{21}$ and let $\fs_{\bC}:=e+\gl_3(\bC)^f$ and $\fs_{\R}:=e+\gl_3(\bR)^f$ be the Slodowy slices.

For the proof we will need the following lemma.

\begin{lem}\label{lem:SlodSlice}
$\fs_{\bC}$ is a slice for the action of $G$ at the point $e$, and for any $x\in \fs_{\bC},$ the Nash manifold $\{g\in G \, |\, gx\in \fs_{\bC}\}$  is connected.
\end{lem}
\begin{proof}
The fact that $\fs_{\bC}$ is a slice for the action of $G$ is standard. Since all the stabilizers of the action of $G$ are connected, in order to prove that $\{g\in G \, |\, gx\in \fs_{\bC}\}$  is connected it is enough to prove that the intersection of any $G$-orbit $\cO$ with $\fs_{\bC}$ is connected. For this it is enough to show that $\bar \cO\cap \fs_{\bC}$ is an irreducible algebraic variety. We divide the proof into two cases.
\begin{enumerate}[{Case} 1]
\item $\bar \cO=\{x \in X \, \vert \, \det(x-\lambda \Id)   = - \lambda^3+\gamma_1\lambda+\gamma_0  \}$ for some fixed $\gamma_0$ and $\gamma_1$. \\
Choose the following coordinates on $\fs_{\bR}$:
\begin{equation}
\fs_{\bR}=\left \{ \left(
  \begin{array}{ccc}
    a & 1 & 0 \\
    b & a & c \\
    d &0 & -2a \\
  \end{array}
\right) \right \}.
\end{equation}
In these coordinates, $\bar\cO \cong \{(a,b,c,d) \, \vert \, 3a^2+b= \gamma_1 \text{ and }cd - 2a^3  + 2ab=\gamma_0\}$. This variety is isomorphic to $\{(a,c,d) \, \vert \,- 8a^3 + 2\gamma_1a + cd=\gamma_0 \}.$ Thus this case follows from the irreducibility of the polynomial $- 8a^3 + 2\gamma_1a + cd-\gamma_0$ for any $\gamma_1,\gamma_2$.

\item $\bar \cO=\{x \in X \, \vert \, (x-\gamma \Id) (x+2\gamma \Id)    = 0  \}$  for some fixed $\gamma$.\\
In the coordinates above $\bar \cO$ is given by the irreducible polynomial $$cd-(2a + \gamma)^2(2a - 2\gamma).$$
\end{enumerate}
\end{proof}

\begin{lem}\label{lem:vi}
The collection of vector fields on $\fs_{\bR}$ strongly tangential to the action of $G$ is generated over $C^{\infty}(\fs_{\bR})$ by the fields $v_1,\dots,v_4$, where
\begin{align*}
v_1(A)= \left(
  \begin{array}{ccc}
    0 & 0 & 0 \\
    0 & 0 & A_{23} \\
    -A_{31} &0 & 0 \\
  \end{array}
\right), \,\, &v_2(A)= \left(
  \begin{array}{ccc}
    0 & 0 & 0 \\
    0 & 0 & -A_{11}A_{23} \\
    A_{11}A_{31} &0 & 0 \\
  \end{array}
\right)\\
v_3(A)= \left(
  \begin{array}{ccc}
    A_{31}/2 & 0 & 0 \\
    -3A_{11}A_{31} & A_{31}/2 & 9A_{11}^2-A_{21} \\
    0 &0 & -A_{31} \\
  \end{array}
\right), \,\, &v_4(A)= \left(
  \begin{array}{ccc}
    -A_{23}/2 & 0 & 0 \\
    3A_{11}A_{23} & -A_{23}/2 &0 \\
    -9A_{11}^2+A_{21}&0 & A_{23} \\
  \end{array}
\right)
\end{align*}

\end{lem}
This lemma is proven by a direct computation.
\begin{proof}[Proof of \Cref{prop:gl3}]
Let
\begin{multline*}
V_{\bC}:=\{\xi \in \cG_{\{e\}}(\fs_{\bC})\, \vert \,  \alp\xi=0 \\ \text{ for any vector field } \alp \text{ on }\fs_{\bC} \text{ strongly tangential to the action of }G\}.
\end{multline*}
and
\begin{multline*}
V_{\bR}:=\{\xi \in \cG_{\{e\}}(\fs_{\bR})\, \vert \,  \alp\xi=0 \\ \text{ for any vector field } \alp \text{ on }\fs_{\bR} \text{ strongly tangential to the action of }\GL_n(\R)\}.
\end{multline*}
By \Cref{lem:SlodSlice,cor:main} $\cG_{\cO}(X)\cong V_{\bC}$. It is easy to see that $V_{\bC}\cong V_{\bR}\otimes V_{\bR}$ as a filtered vector space. By \Cref{lem:vi},
$$V_{\bR}=\{\xi \in \cG_{\{x\}}(\fs_{\bR})\, \vert \, v_i\xi=0 \,\, \forall 1\leq i\leq 4 \}.$$
Choose the following coordinates on $\fs_{\bR}$:
$$\fs_{\bR}=\left \{ \left(
  \begin{array}{ccc}
    a & 1 & 0 \\
    b & a & c \\
    d &0 & -2a \\
  \end{array}
\right) \right \}.$$

In these coordinates we have

\begin{equation*}
v_1=c\frac{\partial}{\partial c}-d\frac{\partial}{\partial d}, \,\, v_2=-av_1,\,\, v_3=\frac{d}{2}\frac{\partial}{\partial a}-3ad\frac{\partial}{\partial b}+(9a^2-b)\frac{\partial}{\partial c}, \,\,v_4=-\frac{c}{2}\frac{\partial}{\partial a}+3ac\frac{\partial}{\partial b}+(b-9a^2)\frac{\partial}{\partial d}
\end{equation*}

Fix a Lebesgue measure on $\fs_{\R}$. It defines the generalized function $\delta_e\in \cG_{\{e\}}(\fs(\R))$. Let  $$ \delta_{ijkl}:=\left( \frac{\partial}{\partial c}\right)^i\left( \frac{\partial}{\partial c}\right)^j\left( \frac{\partial}{\partial c}\right)^k\left( \frac{\partial}{\partial c}\right)^l\delta_e.$$
If one of the indices $i,j,k,l$ is negative we set $ \delta_{ijkl}:=0$.
We have
% (?? since $x^k\frac{\partial}{\partial x}^n\delta=(-1)^kn...(n-k+1)\frac{\partial}
% {\partial x}^{n-k}\delta$)
\begin{align*}
v_1 \delta_{ijkl}&=-(k+1)\delta_{ijkl}+(l+1)\delta_{ijkl},\\
v_3\delta_{ijkl}&=-\frac{l}{2}\delta_{i+1,j,k,l-1}-3il\delta_{i-1,j+1,k,l-1}+9i(i-1)\delta_{i-2,j,k+1,l}+
j\delta_{i,j-1,k+1,l}\\
v_4\delta_{ijkl}&=\frac{k}{2}\delta_{i+1,j,k-1,l}+3ik\delta_{i-1,j+1,k-1,l}-9i(i-1)\delta_{i-2,j,k,l+1}-j\delta_{i,j-1,k,l+1}
\end{align*}
Let $\xi=\sum c_{ijkl}\delta_{ijkl}$
and note that $v_1\xi=0$ if and only if $c_{ijkl}=0\, \forall k\neq l$. Set $\delta_{ijk}:=\delta_{ijkk}$.
Let $\xi=\sum c_{ijk}\delta_{ijk}$ we get
$$v_3\xi=\sum_{i,j\geq 0, k\geq 1} \left( -\frac{k}{2}c_{i-1,j,k}-3(i+1)kc_{i+1,j-1,k}+9(i+2)(i+1)c_{i+2,j,k-1}+(j+1)c_{i,j+1,k-1} \right)\delta_{i,j,k,k-1}$$

$$v_4\xi=\sum_{i,j,k\geq 0} \left( \frac{k+1}{2}c_{i-1,j,k+1}+3(i+1)(k+1)c_{i+1,j-1,k+1}-9(i+2)(i+1)c_{i+2,j,k}-(j+1)c_{i,j+1,k} \right)\delta_{i,j,k,k+1}$$
Here, if one of the indices $i,j,k$ is negative we set $c_{i,j,k}=0$.

We obtain that $V_{\bR}$ is the collection of all finite combinations  $\sum c_{ijk}\delta_{ijk}$ that satisfy
%old computation with confution between gen fun and dist $$c_{i-1,j,k+1}\frac{k+1}{2}-3c_{i+1,j-1,k+1}(i+1)(k+1)+9c_{i+2,j,k}(i+2)(i+1)-c_{i,j+1,k}(j+1)=0$$
$$c_{i-1,j,k+1}\frac{k+1}{2}+3c_{i+1,j-1,k+1}(i+1)(k+1)-9c_{i+2,j,k}(i+2)(i+1)-c_{i,j+1,k}(j+1)=0$$
for all $i,j,k \geq 0.$

Let $F^n$ be the Bruhat filtration on $V_{\bR}$ and $G^l$ be the filtration on $F^n(V_{\bR})$ given by $$G^l(F^n(V_{\bR}))=\left\{\sum c_{ijk}\delta_{ijk} \in F^n(V_{\bR}) \, \vert \, \forall k> l \text{ we have } c_{ijk}=0\right\}.$$
It is easy to compute that $$\dim G^l(F^n(V_{\bR})) - \dim G^{l-1}(F^n(V_{\bR}))=n-2l.$$ Thus $$\dim F^{2m}(V_{\R})=  m(m+1) \text{ and }\dim F^{2m+1}(V_{\R})=  (m+1)^2.$$
Define the power series
$$f(s):=\sum_n s^{n+1}=s/(1-s) \text{ and }g(t):=\sum _n\dim F^n(V_{\R})t^n.$$
Then
$$\sum_mm(m+1)s^m=sf''(s)=2(1-s)^{-3} \text{ and }\sum s^m(m+1)^2=(sf'(s))'=(1+s)(1-s)^{-3}.$$
We get
\begin{multline*}
g(t)=\sum_m(t^{2})^m(m+1)+t\sum_m(t^{2})^m(m+1)^2= 2(1-t^{2})^{-3} +t(1+t^{2})(1-t^{2})^{-3}\\
=(t^3+t+2)(1-t^{2})^{-3}=(t^2-t+2)(1-t)^{-3}(1+t)^{-2}
\end{multline*}
Thus $$\sum _n(\dim F^n(V_{\R})-\dim F^{n-1}(V_{\R}))t^n=(t^2-t+2)(1-t)^{-2}(1+t)^{-2},$$
and hence $$\bar\fG_{x}^{X}(t)=\sum _n(\dim F^n(V_{\bC})-\dim F^{n-1}(V_{\bC}))t^n=(t^2-t+2)^2(1-t)^{-4}(1+t)^{-4}.$$
\end{proof}

\end{document}